\documentclass[reqno,10pt]{amsart}
\usepackage{amsfonts, amsmath, amssymb, amsthm, color}
\allowdisplaybreaks[1]
\newtheorem{thm}{Theorem}[section]
\newtheorem{lemma}{Lemma}[section]
\newtheorem{prop}{Proposition}[section]

\newtheorem{rmk}{Remark}[section]
\theoremstyle{definition}
\numberwithin{equation}{section}
\newcommand{\al}{\alpha}
\newcommand{\tilmu}{{\mu}}
\newcommand{\tilmm}{{{m}}}
\newcommand{\tilrr}{{{r}}}
\newcommand{\lm}{\lambda}
\newcommand{\ts}{\sigma}

 \newcommand{\eps}{\epsilon}

\newcommand{\Pda}{\mathcal P(d\alpha)}

\newcommand{\calP}{\mathcal P}
\newcommand{\fip}{\varphi}

\newcommand{\ee}{\varepsilon}

\newcommand{\Om}{\Omega}

\newcommand{\R}{{\mathbb R}}
\newcommand{\Rd}{\mathbb{R}^2}

\newcommand{\un}{u_n}

\newcommand{\bal}{\gamma}

\newcommand{\barlm}{\lm_{\bal,\ts}}
\newcommand{\cm}{\overline x_\mu}
\begin{document}
\title[A mean field equation with variable intensities]{On the existence and blow-up of solutions for a mean field equation with variable intensities}
\thanks{
${}^*$Corresponding author.\\
This research is partially supported by \textit{Programma di scambi internazionali con universit\`{a} 
ed istituti stranieri per la mobilit\`{a} breve di docenti, ricercatori e studiosi} of Universit\`{a} di Napoli Federico II
and by 
Progetto di ricerca GNAMPA 2015 \textit{Alcuni aspetti di equazioni ellittiche non lineari}}
\author{T.~Ricciardi${}^*$}
\address[Tonia Ricciardi] {Dipartimento di Matematica e Applicazioni, 
Universit\`{a} di Napoli Federico II, Via Cintia, Monte S.~Angelo, 80126 Napoli, Italy}
\email{tonricci@unina.it}
\author{R.~Takahashi}
\address[Ryo Takahashi] {
Division of Mathematical Science, Department of Systems Innovation, Graduate School of Engineering Science, Osaka University, 
Machikaneyama-cho 1-3, Toyonaka-shi, 560-8531, Japan
}
\email{r-takaha@sigmath.es.osaka-u.ac.jp}
\author{G.~Zecca}
\address[Gabriella Zecca] {Dipartimento di Matematica e Applicazioni, 
Universit\`{a} di Napoli Federico II, Via Cintia, Monte S.~Angelo, 80126 Napoli, Italy}
\email{g.zecca@unina.it}
\author{X.~Zhang}
\address[Xiao Zhang]  {
Division of Mathematical Science, Department of Systems Innovation, Graduate School of Engineering Science, Osaka University, 
Machikaneyama-cho 1-3, Toyonaka-shi, 560-8531, Japan
}
\email{zhangx@sigmath.es.osaka-u.ac.jp}
\begin{abstract}
We study an elliptic problem with exponential nonlinearities describing
the statistical mechanics equilibrium of point vortices with variable intensities.
For suitable values of the physical parameters we exclude the existence of blow-up
points on the boundary, we prove a mass quantization property and
we apply our analysis to the construction of minimax solutions.
\bigskip
\par
\noindent
\textit{2010 Mathematics Subject Classification.} 35J91, 35B44, 35J20
\par
\noindent
\textit{Key words and phrases.} Mean field equation, blow-up solutions, turbulent Euler flow.
\end{abstract}
%\subjclass{76B03, 76B47, 35B44}
\date{September 16, 2015}
%\keywords{mean field equation, blow-up solutions, turbulent Euler flow} 
\maketitle
%%%%%%%%%%%%%%%%%%%%%%%%%%%%%%%%%%%%%%%%%%%%%%%%%%%%%%%%%%%%%%%%%%%%%%%%%%%%%%%
\section{Introduction and main result}
Motivated by the theory of hydrodynamic turbulence as developed by Onsager~\cite{ES, On},
we consider the problem:
\begin{equation}
\label{Intro:specialpb}
\left\{
\begin{aligned}
-\Delta u=&\lm\left(\frac{e^u}{\int_\Omega e^u}+ \ts \bal\frac{e^{\bal u}}{\int_\Omega e^{\bal u}}\right)
&&\hbox{in\ }\Omega\\
u=&0&&\hbox{on\ }\partial\Omega,
\end{aligned}
\right.
\end{equation}
where $\lambda, \ts>0$, $\bal\in[-1,1)$ and $\Omega\subset\R^2$ is a smooth bounded domain.
Problem~\eqref{Intro:specialpb} is derived by statistical mechanics arguments under a ``deterministic" assumption on the
point vortex intensities \cite{clmp92, SawadaSuzuki}.
More precisely, the equation derived in \cite{SawadaSuzuki} is given by
\begin{equation}
\label{Intro:generalpb}
\left\{
\begin{aligned}
-\Delta u=&\widetilde\lambda\int_{[-1,1]}\frac{\al e^{\al u}}{\int_{\Omega}e^{\al u}\,dx}\,\Pda&&\hbox{in\ }\Omega\\
u=&0&&\hbox{on\ }\partial\Omega,
\end{aligned}
\right.
\end{equation}
where $u$ is the stream function of the two-dimensional flow, $\calP$ is a Borel probability measure defined on the interval $[-1,1]$
describing the point vortex intensity distribution and $\widetilde\lambda>0$
is a constant related to the inverse temperature.
In the special case $\Pda=\mathcal P_\gamma(d\alpha)$, where
\begin{equation}
\label{Intro:specialP}
\mathcal P_\gamma(d\alpha)=\tau\delta_1(d\al)+(1-\tau)\delta_{\bal}(d\al),
\end{equation}
and $\delta_1(d\alpha),\delta_\gamma(d\alpha)$ denote the Dirac measures concentrated at
the points $1,\gamma\in[-1,1]$, respectively, and $\tau\in(0,1]$,
problem~\eqref{Intro:generalpb} takes the form 
\begin{equation}
\label{pb:withtau}
\left\{
\begin{aligned}
-\Delta u=&\widetilde\lambda\left(\tau\frac{e^u}{\int_\Omega e^u\,dx}+ (1-\tau)\bal\frac{e^{\bal u}}{\int_\Omega e^{\bal u}\,dx}\right)
&&\hbox{in\ }\Omega\\
u=&0&&\hbox{on\ }\partial\Omega.
\end{aligned}
\right.
\end{equation}
Setting 
\begin{equation}
\label{eq:tausigma}
\lm=\widetilde\lambda\tau,
\qquad\sigma=\frac{1-\tau}{\tau},
\end{equation}
problem~\eqref{pb:withtau} reduces to \eqref{Intro:specialpb}.
\par
We observe that taking $\gamma=-1$ in problem~\eqref{Intro:specialpb} we obtain
the sinh-Poisson type problem derived in \cite{PL}:
\begin{equation}
\label{eq:sinhPoisson}
\left\{
\begin{aligned}
-\Delta u=&\lm\left(\frac{e^u}{\int_\Omega e^u}- \ts\frac{e^{-u}}{\int_\Omega e^{-u}}\right)
&&\hbox{in\ }\Omega\\
u=&0&&\hbox{on\ }\partial\Omega,
\end{aligned}
\right.
\end{equation}
which has received a considerable interest in recent years, see \cite{BJMR, grossipistoia, jevnikar, JWYZ, os2006, Ric2007}
and the references therein.
In particular, the blow-up analysis for \eqref{eq:sinhPoisson} has been clarified by geometrical arguments
involving constant mean curvature surfaces
in \cite{JWYZ}. However, such an approach seems difficult to extend to our case.
\par
For $\sigma=0$ problem~\eqref{Intro:specialpb} reduces to the standard
mean field problem
\begin{equation}
\label{eq:standMFE}
\left\{
\begin{aligned}
-\Delta u=&\lm\frac{e^u}{\int_\Omega e^u\,dx}
&&\hbox{in\ }\Omega\\
u=&0&&\hbox{on\ }\partial\Omega,
\end{aligned}
\right.
\end{equation}
which has been extensively analyzed in view of its connections to
differential geometry, physics and biology,
see, e.g., \cite{CSLin}. 
However, even in the ``positive case" $\gamma\in(0,1)$, problem~\eqref{Intro:specialpb}
does not necessarily exhibit the properties of a perturbation of \eqref{eq:standMFE}. 
This fact may be seen, for example, by considering
the optimal constant for the Moser-Trudinger inequality associated to  \eqref{Intro:specialpb},
see \cite{RS, STZ} or the proof of Lemma~\ref{lem:MT} below. In this respect, problem~\eqref{Intro:specialpb} significantly differs from its
``stochastic" version derived 
in \cite{Neri} and recently analyzed in \cite{pistoiaricciardi, RiTa, RZ, RZ2, RicZecprep}.
In fact, our aim in this article is to determine suitable smallness conditions for $|\gamma|$ \textit{and} $\sigma$ 
(see \eqref{def:gammasigma}--\eqref{def:musg} below) which ensure
that the nonlinearity $e^{\gamma u}$ may indeed be treated as ``lower-order"
with respect to the ``principal" term $e^u$.
In particular, under such smallness conditions we prove the mass quantization for 
blow-up solution sequences, we derive an improved Moser-Trudinger inequality and we
consequently obtain an existence result for solutions in the supercritical range $\lambda>8\pi$.
\par
More precisely, for every fixed $\gamma$ satisfying $0<|\gamma|<1/2$ let
\begin{equation}
\label{def:gammasigma}
\sigma_\gamma:=\frac{1-2|\gamma|}{2\gamma^2}
\end{equation}
and
\begin{equation}
\label{def:musg}
\lm_{\sigma,\gamma}:=\min\left\{\frac{16\pi}{1+2\ts\bal^2}, \,\frac{4\pi}{|\bal| (1 +|\bal| \ts )}\right\}.
\end{equation}
Our main result is the following.
\begin{thm}
\label{Intro:thm:existence}
Assume that $\mathbb R^2\setminus\Omega$
has a bounded component containing at least one interior point. 
Fix $0<|\gamma|<1/2$ and $0<\sigma<\sigma_\gamma$.
Then, there exists a solution to Problem \eqref{Intro:specialpb} for every
$\lm\in\left({8\pi},\lm_{\sigma,\gamma}\right)$.
\end{thm}
We note that $8\pi<\barlm<{16\pi}$
whenever $0<|\gamma|<1/2$ and $\sigma\in(0,\sigma_\gamma)$,
see Lemma~\ref{Lemma:alpha} below. 
\par
Finally, we remark that problem~\eqref{Intro:specialpb} shares some similarity in structure with Liouville systems
and Toda-type systems.
Indeed, setting $v_1=G\ast e^u$, $v_2=G\ast e^{\gamma u}$,
$a^{11}=\lambda/\int_\Omega e^u$, $a^{12}=\lambda\sigma\gamma/\int_\Omega e^{\gamma u}$,
$a^{21}=\gamma\lambda/\int_\Omega e^u$, $a^{22}=\gamma^2\sigma\lambda/\int_\Omega e^{\gamma u}$,
we obtain $u=\lambda v_1/\int_\Omega e^u\,dx+\lambda\sigma\gamma v_2/\int_\Omega e^{\gamma u}$
and problem~\eqref{Intro:specialpb} takes the form
$-\Delta v_i=\exp\{\sum_{j=1,2}a^{ij}v_j\}$, $i=1,2$, which is a system of Liouville type, as analyzed in 
\cite{ChanilloKiessling, CSW}.
On the other hand, setting $w_1=u$, $w_2=\gamma u$,
$b^{11}=\lambda/\int_\Omega e^{w_1}$, $b^{12}=\lambda\sigma\gamma/\int_\Omega e^{w_2}$,
$b^{21}=\lambda\gamma/\int_\Omega e^{w_1}$, $b^{22}=\lambda\sigma\gamma^2/\int_e^{w_2}$,
we obtain the system $-\Delta w_i=\sum_{j=1,2}b^{ij}e^{w_j}$ $j=1,2$,
which has a ``Toda-like" structure when $\gamma<0$, see \cite{BJMR} and the references therein.
However, Theorem~\ref{Intro:thm:existence} does not follow directly from the results for
systems of Liouville and Toda type mentioned above, due to the substantially different assumptions for the coefficients
$(a^{ij})$ and $(b^{ij})$, $i,j=1,2$.  
\par
This note is organized as follows. In Section~\ref{sec:blowup} we use Brezis-Merle
estimates \cite{BM} to exclude the existence of blow-up points on the boundary and to
derive a mass quantization property,
for suitably small values of $\gamma$ and $\sigma$.
We note that the exclusion of boundary blow-up points could also be derived by
extending the argument in \cite{RiTa}. Here, we provide a simple \textit{ad hoc}
proof which exploits the smallness assumptions on $|\gamma|$ and $\sigma$.
In Section~\ref{sec:minimax} we derive an improved Moser-Trudinger type inequality.
We prove Theorem~\ref{Intro:thm:existence} by suitably adapting an argument in \cite{DJLW}
and by applying the blow-up results derived in Section~\ref{sec:blowup}.
\subsection*{Notation}
Henceforth, all integrals are taken with respect to the Lebesgue measure. 
We may omit the integration variables if they are clear from the context.
We denote by $C$ a general constant whose actual value may vary from line to line.
 %%%%%%%%%%%%%%%%%%%%%%%%%%%%%%%%%%%%%%%%%%%%%%%%%%%%%%%%%%%%%%%%%%%%%%%%%%%%%%%%%%%
%%%%%%%%%%%%%%%%%%%%%%%%%%%%%%%%%%%%%%%%%%%%%%%%%%%%%%%%%%%%%%%%%%%%%%%%%%%%%%%%%%%
%%%%%%%%%%%%%%%%%%%%%%%%%%%%%%%%%%%%%%%%%%%%%%%%%%%%%%%%%%%%%%%%%%%%%%%%%%%%%%%%%%%
%%%%%%%%%%%%%%%%%%%%%%%%%%%%%%%%%%%%%%%%%%%%%%%%%%%%%%%%%%%%%%%%%%%%%%%%%%%%%%%%%%%
\section{Blow-up results}
\label{sec:blowup}
In this section we show that for suitably small values of $|\gamma|$ and $\sigma$ the blow-up analysis
for problem~\eqref{Intro:specialpb} is similar to the blow-up analysis for\ the standard mean field equation
\eqref{eq:standMFE}.
We first exclude the existence of boundary blow-up points. Then, we prove a mass quantization property.
\par
More precisely, let $(u_n,\lm_n)$ be a solution sequence for \eqref{Intro:specialpb} with $\lm_n\to\lm_0\ge0$.
We define
\[
\mathcal S_\pm=\{x_0\in\overline\Omega:\ \exists x_n\to x_0\ \mbox{such that\ }u_n(x_n)\to\pm\infty\}
\]
and we set $\mathcal S=\mathcal S_+\cup\mathcal S_-$.
\subsection{Boundary blow-up exclusion}
In the case where $\gamma>0$ the boundary blow-up is readily excluded in view of
the moving plane argument in \cite{GNN}, p.~223.
Therefore, throughout this subsection, we consider the ``asymmetric sinh-case" of  \eqref{Intro:specialpb},
namely
\begin{equation}
\label{Bup-exclusion:specialpb}
\left\{
\begin{aligned}
-\Delta u=&\lm\left(\frac{e^u}{\int_\Omega e^u}-\ts|\bal|\frac{e^{-|\bal|u}}{\int_\Omega e^{-|\bal|u}}\right)
&&\hbox{in\ }\Omega\\
u=&0&&\hbox{on\ }\partial\Omega.
\end{aligned}
\right.
\end{equation}
We make the following assumption:
\begin{equation}
\label{assumpt:BM}
\lm (1+\ts|\bal|)<\frac{4\pi}{|\gamma|}.
\end{equation}
In this subsection we show the following.
\begin{prop}
\label{prop:bdryblowupexcl}
Let $(u_n,\lm_n)$ be a solution sequence for problem~\eqref{Bup-exclusion:specialpb} with $\lm_n\to\lm_0\ge0$
and assume that $\lm_0$ satisfies \eqref{assumpt:BM}.
Then, $\mathcal S\cap\partial\Omega=\emptyset$.
\end{prop}
We first reduce problem \eqref{Bup-exclusion:specialpb} to a mean field type problem with smooth weight function.
Let $G=G(x,y)$ be the Green's function defined for $x,y\in\Omega$ by
\begin{equation*}
\label{def:Green}
\left\{
\begin{aligned}
-\Delta G(\cdot,y)=&\delta_y
&&\hbox{in\ }\Omega\\
G(\cdot,y)=&0&&\hbox{on\ }\partial\Omega.
\end{aligned}
\right.
\end{equation*}
Let $u:=u_+-u_-$, where  
\begin{equation*}
\begin{split}
&u_+= G*\lm \frac{e^u}{\int_\Omega e^u}\\
&u_-= G*\lm\ts|\bal|\frac{e^{-|\bal|u}}{\int_\Omega e^{-|\bal|u}}.\\
\end{split}
\end{equation*}
We observe that 
\begin{equation*}
\label{Bup-exclusion:equiv-specialpb}
\left\{
\begin{aligned}
-\Delta u_+=&\lm \frac{h(x)e^{u_+}} {\int_\Omega h(x) e^{u_+}}
&&\hbox{in\ }\Omega\\
u_+=&0&&\hbox{on\ }\partial\Omega,
\end{aligned}
\right.
\end{equation*}
where $h(x)=e^{-u_-}$ satisfies $\|h\|_{C^{1,\alpha}(\bar\Omega)} \leq C$, $h\equiv 1$ on $\partial \Om$. 
In fact, we have
\begin{equation*}
\left\{
\begin{aligned}
-\Delta u_-= &\lm\ts|\bal|\frac{e^{-|\bal|u}}{\int_\Omega e^{-|\bal|u}}
&&\hbox{in\ }\Omega\\
u_-=&0&&\hbox{on\ }\partial\Omega,
\end{aligned}
\right.
\end{equation*}
where $\lm \ts|\bal|\frac{e^{-|\bal|u}}{\int_\Omega e^{-|\bal|u}}$ is $L^q$-bounded for some $q>1$. 
To see this fact, 
recall from \cite{BM} that if $u$ satisfies:
\begin{equation*}
\left\{
\begin{aligned}
-\Delta u=&f&&\hbox{in\ }\Omega\\
u=&0&&\hbox{on\ }\partial\Omega,
\end{aligned}
\right.
\end{equation*} 
for some $f\in L^1(\Om),$ then for any small $\delta >0$ we have
\begin{equation}
\label{eq:BMestimate}
\int_\Omega\exp\{\frac{(4\pi-\delta)}{\|f\|_{L^1(\Omega)}}|u|\}\le\frac{4\pi^2}{\delta}(\mathrm{diam}\Omega)^2.
\end{equation}
Hence, by elliptic estimates, $\|u_-\|_{L^\infty(\Om)} \leq C$.
Now we write the equation for $u_-$ in the form 
\begin{equation}
\label{eq:uminusbdd}
\left\{
\begin{aligned}
-\Delta u_-= &\lm \ts|\bal|\frac{e^{|\bal|u_- -|\bal|u_+}}{\int_\Omega e^{-|\bal|u}}
&&\hbox{in\ }\Omega\\
u_-=&0&&\hbox{on\ }\partial\Omega
\end{aligned}
\right.
\end{equation}
and we observe that since $u_+\geq 0$ we have $e^{-|\bal|u_+}\leq 1.$  
Hence, the right hand side in \eqref{eq:uminusbdd} is $L^\infty(\Omega)$-bounded. 
It follows that $\|u_-\|_{W^{2,p} (\Omega)} \leq C$ for every $p\in (1,+\infty)$. 
In particular, $\|u_-\|_{C^{1,\alpha}(\overline\Omega)} \leq C$.
\begin{proof}[Proof of Proposition~\ref{prop:bdryblowupexcl}]
We adapt an argument of \cite{GNN} p.~223 to our case.
Let $x_0\in\partial\Omega$ and let $D_r$ be a closed disc touching $\overline\Omega$ only at $x_0$.
For convenience we assume $D_r=D(0,r)$ and $x_0=(r,0)$.
Then, the inversion mapping $x\mapsto y=r^2x/|x|^2$ fixes the boundary of $D_r$ and maps $\Omega$ to a region 
$y(\overline\Omega)$ contained inside $D_r$.
Setting $v(y)=u^+(x)$, recalling that
\begin{equation}
\label{eq:invlin}
D_x=\frac{r^2}{|y|^4}\left(
\begin{aligned}
&-y_1^2+y_2^2&&-2y_1y_2\\
&-2y_1y_2&&y_1^2-y_2^2
\end{aligned}
\right),
\qquad\qquad D_x^TD_x=\frac{r^4}{|y|^4}I,
\end{equation}
where $I$ denotes the identity mapping, we obtain the following equation for $v$:
\begin{equation*}
\label{eq:vinverion}
\Delta v+\rho\frac{r^4}{|y|^4}h(x(y))e^v=0
\qquad\mbox{in\ }y(\overline\Omega).
\end{equation*}
\underline{Claim.}
For every $x_0\in\partial\Omega$ there exist $r(x_0)>0$, $\delta(x_0)>0$ and $\delta'(x_0)>0$ such that the function
\[
H(y)=\frac{r^4}{|y|^4}h(x(y))
\]
is decreasing in $y_1$-direction in the set
\[
\left\{
\frac{r(x_0)}{1+\delta(x_0)}\le|y|\le r(x_0),\ y_1>0,\ |y_2|\le\delta'(x_0)
\right\}\subset y(\overline\Omega),
\]
provided that $r\le r(x_0)$.
\\
Proof of Claim.
In view of \eqref{eq:invlin} we compute:
\begin{equation*}
\begin{aligned}
&\partial_{y_1}\frac{1}{|y|^4}=-\frac{4y_1}{|y|^6}\\
&\partial_{y_1}h(x(y))=\frac{r^2}{|y|^4}\Big\{(\partial_{x_1}h(x(y)))(-y_1^2+y_2^2)+(\partial_{x_2}h(x(y)))(-2y_1y_2)\Big\}
\end{aligned}
\end{equation*}
and therefore
\begin{equation*}
\begin{aligned}
\frac{1}{r^4}\partial_{y_1}H(y)=&\frac{1}{|y|^6}
\Big\{
-4y_1h(x(y))+\frac{r^2}{|y|^2}[(\partial_{x_1}h(x(y)))(-y_1^2+y_2^2)+(\partial_{x_2}h(x(y)))(-2y_1y_2)]
\Big\}\\
=&\frac{1}{|y|^6}
\Big\{
-y_1[4h(x(y))+\frac{r^2}{|y|^2}(\partial_{x_1}h)y_1]+\frac{r^2}{|y|^2}y_2^2(\partial_{x_1}h)
+\frac{r^2}{|y|^2}(\partial_{x_2}h)(-2y_1y_2)
\Big\}.
\end{aligned}
\end{equation*}
We estimate, for $|y|\ge r/(1+\delta)$, $|y_2|<\delta'$:
\begin{equation*}
\begin{aligned}
&\left|\frac{r^2}{|y|^2}(\partial_{x_1}h)y_1\right|\le(1+\delta)^2\|h\|_{C^1(\overline\Omega)}r\\
&\left|\frac{r^2}{|y|^2}y_2^2(\partial_{x_1}h)\right|\le(1+\delta)^2\|h\|_{C^1(\overline\Omega)}(\delta')^2\\
&\left|\frac{r^2}{|y|^2}(\partial_{x_2}h)(-2y_1y_2)\right|\le2(1+\delta)^2\|h\|_{C^1(\overline\Omega)}r\delta'.
\end{aligned}
\end{equation*}
By choosing $r=r(x_0)$ sufficiently small, we achieve
\[
4h(x(y))+\frac{r^2}{|y|^2}(\partial_{x_1}h)y_1\ge2.
\]
Then, for $\delta'$ sufficiently small, we have
\begin{equation*}
\begin{aligned}
-y_1[4h(x(y))+\frac{r^2}{|y|^2}(\partial_{x_1})y_1]&+\frac{r^2}{|y|^2}y_2^2(\partial_{x_1}h)
+\frac{r^2}{|y|^2}(\partial_{x_2}h)(-2y_1y_2)\\
\le&-2y_1+(1+\delta)^2\|h\|_{C^1(\overline\Omega)}(\delta')^2+2(1+\delta)^2\|h\|_{C^1(\overline\Omega)}r\delta'\\
\le&-\frac{r}{2}<0.
\end{aligned}
\end{equation*}
Now the argument in \cite{GNN} concludes the proof.
\end{proof}
\subsection{Mass quantization}
In view of Proposition~\ref{prop:bdryblowupexcl} we have $\mathcal S\cap\partial\Omega=\emptyset$.
Therefore, by local blow-up results from \cite{ORS,RZ} we know that setting
\[
\tilmu_1(dx):=\lm\frac{e^u}{\int_\Omega e^u}\,dx,\qquad\qquad
\tilmu_{\bal}(dx):=\lm\frac{e^{\bal u}}{\int_\Omega e^{\bal u}}\,dx
\]
we have 
\begin{equation}
\label{eq:mulimits}
\begin{aligned}
\tilmu_1(dx)&\stackrel{\ast}{\rightharpoonup}\sum_{p\in\mathcal S}\tilmm_1(p)\delta_p(dx)+\tilrr_1(x)\,dx\\
\tilmu_{\bal}(dx)&\stackrel{\ast}{\rightharpoonup}\sum_{p\in\mathcal S}\tilmm_{\bal}(p)\delta_p(dx)+\tilrr_{\bal}(x)\,dx.
\end{aligned}
\end{equation}
\begin{lemma}
At every fixed $p\in\mathcal S$ we have the quadratic identity:
\begin{equation}
\label{eq:quadid}
8\pi( \tilmm_1(p)+\ts \tilmm_{\bal}(p))=\left( \tilmm_1(p)+\ts\bal \tilmm_{\bal}(p)\right)^2.
\end{equation}
\end{lemma}
\begin{proof}
We recall from \cite{ORS} that if
$(u_k,\widetilde\lambda_k)$ is a solution sequence for 
\eqref{pb:withtau} with
\[
\widetilde\lambda\frac{e^{u_k}}{\int_\Omega e^{u_k}}\stackrel{*}{\rightharpoonup}
\sum_{p\in\mathcal S}\widetilde m_1(p)\delta_p(dx)+\widetilde r_1(x),
\qquad
\widetilde\lambda\frac{e^{\gamma u_k}}{\int_\Omega e^{\gamma u_k}}\stackrel{*}{\rightharpoonup}
\sum_{p\in\mathcal S}\widetilde m_\gamma(p)\delta_p(dx)+\widetilde r_\gamma(x),
\]
where $\delta_p(dx)$ denotes the Dirac mass centered at $p\in\Omega$,
then the following relation holds:
\[
8\pi(\tau\widetilde m_1(p)+(1-\tau)\widetilde m_\gamma(p))
=\left(\tau\widetilde m_1(p)+(1-\tau)\gamma\widetilde m_\gamma(p)\right)^2,
\]
for every $p\in\mathcal S$.
In view of \eqref{eq:tausigma} we have $\tau\widetilde m_1(p)=m_1(p)$
and $(1-\tau)\widetilde m_\gamma(p)=\sigma m_\gamma(p)$.
Hence, we derive the asserted identity.
Alternatively, we may derive identity~\eqref{eq:quadid}
by applying the Pohozaev identity in a standard way.
\end{proof}
\begin{lemma}
\label{lem:intboundbelow}
Let $u_n$ be a solution sequence for \eqref{Intro:specialpb}.
For any $\bal\in[-1,1]$ we have
\begin{equation*}
\label{eq:intboundbelow}
\int_\Omega e^{\bal u_n}\ge c_0>0.
\end{equation*}
\end{lemma}
\begin{proof}
If $\bal>0$, we have $u_n\ge0$ in $\Omega$ by the maximum principle and therefore
\[
\int_\Omega e^{\bal u_n}\ge|\Omega|>0.
\]
Therefore, we assume $\bal<0$.
We note that since $\|u_n\|_{W_0^{1,q}(\Omega)}\le C$ for any $q\in[1,2)$, there 
exists $u_0\in W_0^{1,q}(\Omega)$ such that $u_n\to u_0$ weakly in $W_0^{1,q}(\Omega)$,
strongly in $L^p(\Omega)$ for any $p\ge1$ and a.e.\ in $\Omega$.
In view of Fatou's lemma, we derive
\[
\liminf_{n\to\infty}\int_\Omega e^{\bal u_n}\ge\int_\Omega e^{\bal u_0}>0.
\]
\end{proof}
\begin{prop}[Mass quantization] 
\label{Prop:masszero} Let $(u_n,\lm_n)$ be a solution sequence for \eqref{Intro:specialpb} with $\lm_n\to\lm_0.$
Assume that $|\gamma|<1/2$ and $\sigma\in(0,\sigma_\gamma)$, where $\sigma_\gamma$ is defined in \eqref{def:gammasigma}.
Moreover, assume that
\begin{equation}
\label{lambdan}
{8\pi}<\lm_0<\frac{4\pi}{|\bal| (1 +|\bal| \ts)}.
\end{equation}
Then, % for every $p\in\mathcal S$ 
we have
$\tilmm_{\bal}(p)=0$ and consequently $\tilmm_1(p)=8\pi$, $\tilrr_1\equiv0$ and $\lm_0\in8\pi\mathbb N$. 
\end{prop}
\begin{proof} 
Throughout this proof we omit the subscript $n$. 
Similarly as above, in view of \eqref{eq:BMestimate} with $f=\lm(\frac{e^u}{\int_\Omega e^u}+\ts \bal\frac{e^{\bal u}}{\int_\Omega e^{\bal u}})$,
$\|f\|_{L^1(\Omega)}\le \lm(1+\ts |\bal|)$, we have that
$\|e^{\bal u}\|_{L^q(\Omega)}$ is bounded if 
$1<q<4\pi/[\lm|\bal| (1+\ts |\bal|)]$. 
The existence of such a $q>1$ follows from \eqref{lambdan}.
Moreover, since by assumption we have $\lm >{8\pi},$ we derive that necessarily 
\[
{8\pi}<\frac{4\pi}{|\bal| ( 1 +\ts |\bal|)}.
\]
This inequality holds in view of the assumption $\sigma\in(0,\sigma_\gamma)$. 
Therefore we have that  $1<4\pi/[\lm |\bal|(1 + \ts |\bal|)]$ and 
$\|e^{\bal u}\|_{L^q(\Omega)}$ is bounded for some $q>1$.
Hence, $\tilmm_{\bal}(p)=0$. Now \eqref{eq:quadid} implies $\tilmm_1(p)=8\pi$.
We decompose $u=w_1+w_2$, with
\begin{equation*}
\left\{
\begin{aligned}
-\Delta w_1=&\lm\frac{e^u}{\int_\Omega e^u}&&\hbox{in\ }\Omega\\
w_1=&0&&\hbox{on\ }\partial\Omega,
\end{aligned}
\right.
\end{equation*} 
and
\begin{equation*}
\left\{
\begin{aligned}
-\Delta w_2=&\ts \bal\lm\frac{e^{\bal u}}{\int_\Omega e^{\bal u}}&&\hbox{in\ }\Omega\\
w_2=&0&&\hbox{on\ }\partial\Omega.
\end{aligned}
\right.
\end{equation*} 
Then, setting
\[
\phi=\ts \bal\lm\frac{e^{\bal u}}{\int_\Omega e^{\bal u}}
\]
we have $\|\phi\|_{L^q(\Omega)}\le C$ for some $q>1$ and therefore $\|w_2\|_{L^\infty(\Omega)}\le C$.
It follows that $e^u=he^{w_1}$ with $h=e^{w_2}\ge e^{\inf_\Omega w_2}\ge e^{-C(\lm)}>0$.
Moreover,
\[
w_1\to G\ast ( \tilmm_1(p)\delta_p+\tilrr_1)=4\log\frac{1}{|x-p|}+\omega+ G\ast \tilrr_1
\]
with $\omega$ smooth in the closure of a neighbourhood $U$ of $p$.
Therefore, by Fatou's lemma:
\[
\liminf\int_\Omega e^u\ge\int_\Omega\liminf e^u\ge e^{-C(\lm)}\int_U e^\omega \frac{dx}{|x-p|^4}=+\infty.
\]
Hence $\tilrr_1\equiv0$ since $u $ is locally uniformly bounded in $\Omega\setminus \mathcal S$.
\par\noindent
Now, the first equation in \eqref{eq:mulimits} implies the mass quantization $\lm\in8\pi\mathbb N$.
\end{proof}
%%%%%%%%%%%%%%%%%%%%%%%%%%%%%%%%%%%%%%%%%%%%%%%%%%%%%%%%%%%%%%%%%%%%%%%%%%%%%%%%%%%%%%%%%%%%
%%%%%%%%%%%%%%%%%%%%%%%%%%%%%%%%%%%%%%%%%%%%%%%%%%%%%%%%%%%%%%%%%%%%%%%%%%%%%%%%%%%%%%%%%%%%
%%%%%%%%%%%%%%%%%%%%%%%%%%%%%%%%%%%%%%%%%%%%%%%%%%%%%%%%%%%%%%%%%%%%%%%%%%%%%%%%%%%%%%%%%%%%
%%%%%%%%%%%%%%%%%%%%%%%%%%%%%%%%%%%%%%%%%%%%%%%%%%%%%%%%%%%%%%%%%%%%%%%%%%%%%%%%%%%%%%%%%%%%
\section{Proof of Theorem~\ref{Intro:thm:existence}}
\label{sec:minimax}
In this section we prove Theorem~\ref{Intro:thm:existence} by suitably adapting a variational argument due to \cite{DJLW}.
The variational functional for Problem~\eqref{Intro:specialpb} is given by:
\begin{equation*}
J_\lm(u)=\frac{1}{2}\int_\Omega|\nabla u|^2\,dx-\lm\,\ln\int_\Omega e^u\,dx
-\lm\ts\,\ln\int_\Omega e^{\bal u}\,dx.
\end{equation*}
%%%%%%%%%%%%%%%%%%%%%%%%%%%%%%%%%%%%%%%%%%%%%%%%%%%%%%%%%%%%%%%%%%%%%%%%%%%%%%%%%%%
\subsection{An improved Moser-Trudinger inequality}
We observe that the standard well-known improved Moser-Trudinger inequality
\cite{ChLi} readily implies an improved inequality for $J_\lm$.
For any fixed $a_0,d_0>0$ we consider the set
\[
\mathcal A_{a_0,d_0}:=\left\{u\in H_0^1(\Omega):\ \exists \Omega_1,\Omega_2\subset\Omega\ \mathrm{s.t.}\
\mathrm{dist}(\Omega_1,\Omega_2)\ge d_0\ \hbox{and\ }
\frac{\int_{\Omega_i}e^u\,dx}{\int_\Omega e^u\,dx}\ge a_0,\ i=1,2\right\}.
\]
\begin{lemma}[Improved Moser-Trudinger Inequality]
\label{lem:IMT}
The functional~$J_\lm$ is bounded from below on $\mathcal A_{a_0,d_0}$
if
\begin{equation}\label{lam-ImpMT}
\lm<\frac{16\pi}{1+2\ts\bal^2}.
\end{equation}
\end{lemma}
\begin{proof}
For $\eps\in(0,1)$ to be fixed later we decompose:
\begin{equation}\label{strateg}
\begin{split}
J_\lm(u)=&(1-\eps)\{\frac{1}{2}\int_\Omega|\nabla u|^2\,dx-\frac{\lm }{1-\eps}\ln\int_\Omega e^u\}\\
&\qquad+\frac{\eps}{\bal^2}\{\frac{1}{2}\int_\Omega|\nabla \bal u|^2\,dx-\frac{\lm \ts \bal^2}{\eps}\ln\int_\Omega e^{\bal u}\}\\
:=&K^1(u)+K^\bal(u)\\
\end{split}
\end{equation}
In view of the Improved Moser-Trudinger inequality, the functional~$K^1$ is bounded below on $\mathcal A_{a_0,d_0}$
if 
\begin{equation}
\label{eq:IMT1}
\frac{\lm}{1-\eps}<16\pi.
\end{equation}
On the other hand, the functional~$K^\bal$ is bounded below on $H_0^1(\Omega)$
if 
\begin{equation}
\label{eq:IMT2}
\frac{\lm\ts \bal^2}{\eps}\le 8\pi.
\end{equation}
Considering \eqref{eq:IMT1} and \eqref{eq:IMT2} we can take a suitable $\eps\in(0,1)$ satisfying
\[
\eps<1-\frac{\lm}{16\pi}
\qquad\hbox{and\ }\qquad\eps\ge\frac{\lm\ts \bal^2}{8\pi}
\]
if 
\[
\frac{\lm\ts \bal^2}{8\pi}<1-\frac{\lm}{16\pi}
\]
which is equivalent to \eqref{lam-ImpMT}.
\end{proof}
\begin{rmk}
Actually, we expect boundedness below of $J_\lm$ for all $\lm\in(8\pi,16\pi)$.
\end{rmk}
%%%%%%%%%%%%%%%%%%%%%%%%%%%%%%%%%%%%%%%%%%%%%%%%%%%%%%%%%%%%%%%%%%%%%%%%%%%%%%%%%%%
\begin{lemma} 
\label{lem:MT} 
For every 
$0<|\gamma|<1/2$ and for every $0<\sigma<(1-2|\gamma|)/(2\gamma^2)=\sigma_\gamma$, the functional
$J_\lm$ is bounded below on $H_0^1(\Omega)$ if and only if 
$\lm \leq {8\pi}$.
\end{lemma}
\begin{proof}
We rewrite 
\begin{equation}
\label{eq:Jdef}
J_\lm (u)= \tilde{J}_{\tilde\lambda}(u)=\frac{1}{2}\int_\Omega|\nabla u|^2\,dx-{\tilde\lambda}\tau\,\ln\int_\Omega e^u\,dx
-{\tilde\lambda}(1-\tau)\ln\,\int_\Omega e^{\bal u}\,dx.
\end{equation}
where:
\[
\ts=\frac{1-\tau}{\tau}\qquad \mbox{ and }\qquad {\tilde\lambda}= \frac{\lm}{\tau}\qquad \tau\in(0,1].
\]
%\underline{Claim.}
%Let $\gamma\in[-1,1/2)$.
%Suppose $\sigma\in(0,\gamma^{-2}(1-2\gamma))$ if $\gamma>0$ and $\sigma\in(0,\gamma^{-2})$ if $\gamma<0$.
%Then, $\tilde J_{\tilde\lambda}$ is bounded below if and only if 
%${\tilde\lambda} \leq \frac{8\pi}{\tau}$ whenever $\bal\in(-\sqrt{\frac{\tau}{1-\tau}}, \frac{\sqrt\tau}{1+\sqrt{\tau}})$.
%\\
%Proof of Claim. 
We use a result from \cite{RS} for the functionals of the form
\begin{equation*}
J_{\tilde\lambda}^{\mathcal P}(u)=\frac{1}{2}\int_\Omega|\nabla u|^2\,dx-{\tilde\lambda} \int_I \left( \log \int_\Omega e^{\alpha u}\,dx\right)\Pda,
\end{equation*}
$u\in H_0^1(\Omega)$.
Note that $J_{\tilde\lambda}^{\mathcal P}$ is the Euler-Lagrange funtional for problem~\eqref{Intro:generalpb}.
In view of Theorem~4 in \cite{RS} (see also \cite{ShafrirWolansky}) we have that $\tilde J_{\tilde\lambda}$ 
is bounded below if and only if 
${\tilde\lambda}\leq\bar \lambda^\mathcal P $ where
\[
\bar \lambda^\mathcal P  =8\pi\inf \left\{\frac{\mathcal P(K_\pm) }{(\int_{K_\pm} \alpha \Pda)^2} , K_\pm \subset I_\pm\cap\mathrm{supp}\mathcal P \right\},
\]	
$I_+=[0,1]$, $I_-=[-1,0)$
and $\mathcal P=\mathcal P_\gamma$ is defined by \eqref{Intro:specialP},
i.e., $\mathcal P_\gamma(d\alpha)=\tau\delta_1(d\al)+(1-\tau)\delta_{\bal}(d\al)$.
\par
Assume $\bal\geqslant 0$. In this case, we have
\begin{equation*}
\tau\frac{\mathcal P_\gamma(K) }{(\int_{K} \alpha\,\mathcal P_\gamma(d\alpha))^2} =\left\{
\begin{array}{ll}
1,
&\hbox{if\ }K=\{1\}\\
&\\
\frac{\tau}{{\bal^2 (1-\tau)}}=\frac{1}{\sigma\gamma^2},&\hbox{if\ }K=\{\bal\}\\
&\\
\frac{\tau}{(\tau +\bal(1-\tau))^2}=\frac{1+\sigma}{(1+\sigma\gamma)^2},&\hbox{if\ }K=\{\bal,1\}.\\
\end{array}
\right.
\end{equation*}
Hence, if $\gamma>0$ we have $\tau\bar \lambda^\mathcal P =8\pi$ whenever $0<\sigma\leqslant \frac{1-2\gamma}{2\gamma^2}=\sigma_\gamma$.
\\
Analogously,  for $\bal<0$ we have 
\begin{equation*}
\tau\frac{\mathcal P_\gamma(K) }{(\int_{K}\alpha\,\mathcal P_\gamma(d\alpha))^2}=
\left\{
\begin{array}{ll}
1,
&\hbox{if\ }K=\{1\}\\
&\\
\frac{\tau}{{\bal^2(1-\tau)}}=\frac{1}{\sigma\gamma^2},&\hbox{if\ }K=\{\bal\}.
\\
\end{array}
\right.
\end{equation*}
Hence, if $\gamma<0$ we have that $\tau\bar\lambda^\mathcal P=8\pi$ if $0<\sigma<\sigma_\gamma$. 
%This establishes the claim.
%See also \cite{STZ}.
%\\
%Now we observe that:
%\[
%\sqrt{\frac{\tau}{1-\tau}}=\frac{1}{\sqrt \ts}\qquad \mbox{ and }\qquad  \frac{ \sqrt{\tau} }{1+\sqrt \tau} =\frac{1}{1+\sqrt{1+\ts}}
%\]
%Recalling the definition of $\bal_\ts$ given by \eqref{def:gammasigma}, it is clear that $(-\bal_\ts,\bal_\ts)\subseteq  (-\frac{1}{\sqrt \ts}, \frac { 1}{1+\sqrt{1+\ts}})$.
\end{proof}
\begin{lemma} 
\label{Lemma:alpha} 
Let $0<|\gamma|<1/2$ and let $0<\sigma<\sigma_\gamma$, where $\sigma_\gamma$ is defined in \eqref{def:gammasigma}.
Then, we have
${8\pi}<\barlm \leq {16\pi}$, where $\barlm$ is defined in \eqref{def:musg}.
\end{lemma}
\begin{proof}
The upper bound is clear. Therefore, we only  prove the lower bound  $\barlm>{8\pi}$. 
We readily check that
\begin{equation*}
\label{alpha1}
\frac{16 \pi}{1+ 2\ts \bal^2}>{8\pi} \qquad \mbox{if and only if}\qquad \sigma<\frac{1}{2\gamma^2}
\end{equation*}
and
\begin{equation*}
\label{alpha2}
\frac{4\pi}{|\bal| (1+\ts |\bal|)}>{8\pi} \qquad  \mbox{if and only if}\qquad 
\sigma<\sigma_\gamma=\frac{1-2|\gamma|}{2\gamma^2}<\frac{1}{2\gamma^2}.
\end{equation*}
The claim follows.
\end{proof}
For every $u\in H^1_0(\Om)$ we consider the measure:
\begin{align*}
\mu_u=\frac{e^{ u}}{\int_{\Om}e^{ u} dx}\,dx
\in\mathcal M(\Om)
\end{align*}
and the corresponding ``center of mass":
\begin{align*}
\cm(u)=\int_\Om x\,d\mu_u\in\Rd.
\end{align*}
\begin{lemma}
\label{baricentri}Let $\lm\in (8\pi, \frac{16\pi}{ 1+2\ts\bal^2})$  and let $\{u_n\}\subset H^1_0(\Om)$ 
be a sequence such that $J_\lm(u_n)\rightarrow -\infty$ and $\cm(\un)\to x_0\in \R^2. $
Then, $x_0\in\overline\Omega$ and 
\begin{equation*}
\mu_{\un}\rightharpoonup \delta_{x_0}\quad \mbox{weakly* in }\mathcal C(\bar \Om)'.
\end{equation*}
\end{lemma}
\begin{proof}
For every fixed $ r>0$ we denote by $\mathcal Q_n(r)$ the concentration function of $\mu_n$, i.e.
\[
\mathcal Q_n(r)= \sup_{x\in \Om} \int_{B(x,r)\cap \Om} \mu_n.
\]
For every $n$, there exists $\tilde x_n\in\overline\Om$ such that
\[
\mathcal Q_n(r/2) =\int_{B(\tilde x_n, r/2)\cap \Om} \mu_n.
\]
Upon taking a subsequence, we have that $\tilde x_n\to\tilde  x_0\in\overline\Om$.

Now, let us set
\[
\Om_1^n =B(\tilde x_n, r/2)\cap \Om\qquad \mbox{ and } \qquad \Om_2^n = \Om \setminus B(\tilde x_n,r),
\]
so that 
\[
\mathrm{dist}(\Om_1^n, \Om_2^n)\geqslant r/2.
\]  

Since $J_\lm (u_n) \rightarrow -\infty $ and since $\lm<\frac{16\pi}{ 1+2\ts\bal^2}$,
in view of \eqref{strateg} necessarily we have $K^1(u_n)\to-\infty$.
Therefore, in view of the standard Improved Moser-Trudinger inequality \cite{ChLi}, we conclude that
\[
\min\{\mu_n(\Om_1^n), \mu_n(\Om_2^n)\}\rightarrow0.
\]
In particular, $\min\{\mathcal Q_n (r/2),1-\mathcal Q_n(r)\}\leqslant\min\left( \mu_n(\Om_1^n), \mu_n(\Om_2^n)\right)\rightarrow0.$

On the other hand, for every fixed $r>0$ let $k_r\in\mathbb N$ be such that $\Om$ is covered
by $k_r$ balls of radius $r/2$. Then, $1=\mu_n(\Om)\leqslant k_rQ_n (r/2)$, so that $Q_n (r/2)\geqslant k_r^{-1}$
for every $n$. We conclude that necessarily $\mathcal Q_n(r)\to1$ 
as $n\to\infty$. Since $r>0$ is arbitrary, we derive in turn that
$1-\mathcal Q_n(r/2)=\mu_n(\Om\setminus B(\tilde x_n, r/2))\to0$
as $n\to\infty$. That is, $\mu_{\un}\rightharpoonup\delta_{\tilde  x_0}$.
It follows that $\cm(\un)=\int_\Om x\,d\mu_n\to \tilde x_0=x_0\in \bar\Om$,
as asserted.
\end{proof}
At this point, in order to prove Theorem~\ref{Intro:thm:existence}, we shall adapt a construction in \cite{DJLW}.
Let $\Gamma_1\subset\Om$ be a non-contractible curve 
which exists since $\Om$ is  non-simply connected. Let $\mathbb D =\{(r,\theta):0\leqslant r<1,\ 0\leqslant \theta <2\pi\}$ be the unit disc. 
Define
\begin{equation*}
\mathcal D_\lm:=\left\{
h\in C(\mathbb D,H_0^1(\Omega))\ \hbox{s.t.:}\qquad
\begin{aligned}
&(i)\qquad\lim_{r\to1}\sup_{\theta\in[0,2\pi)}J_\lm(h(r,\theta))=-\infty\\
&(ii)\qquad \cm(h(r,\theta))\ \hbox{can be extended continuously to  }\overline{\mathbb D}\\
&(iii)\qquad \cm(h(1,\cdot))\ \hbox{ is one-to-one }  \hbox { from }\partial\mathbb D\ \hbox{onto\ }\Gamma_1
\end{aligned}
\right\}
\end{equation*}
\begin{lemma}
\label{lemma-dom-1}
For every $\lm \in ({8\pi},{16\pi})$ the set $\mathcal D_\lm$ is non-empty.
\end{lemma}
\begin{proof}
Let $\gamma_1(\theta):[0,2\pi) \rightarrow \Gamma_1$ be a parametrization of $\Gamma_1$ and let 
$\ee_0>0$ be sufficiently small so that $B(\gamma_1(\theta),\ee_0)\subset \Om$. 
Let $\varphi_\theta(x)=\ee_0^{-1}(x-\gamma_1(\theta))$ so that $\varphi_\theta(B(\gamma_1(\theta),\ee_0))=B(0,1)$. 
We define ``truncated Green's function'':
\begin{equation*}
V_{r}(X)=\left\{
\begin{split}
&4\log\frac1{1-r} \qquad \mbox{ for }X\in B(0,1-r)\\
&4\log\frac1{|X|}\qquad \mbox{ for }X\in B(0,1)\setminus B(0,1-r)\\
\end{split}\right.
\end{equation*}
and
\begin{equation*}
v_{r,\theta}(x)=\left\{
\begin{split}
&0\qquad \mbox{ for }x\in \Om \setminus B(\gamma_1(\theta),\ee_0))\\
&V_{r}(\varphi_\theta(x))\quad \mbox{ for }x\in  B(\gamma_1(\theta),\ee_0).\\
\end{split}\right.
\end{equation*}
We set 
\begin{equation}
\label{def:h}
h(r,\theta)(x)=v_{r,\theta}(x),\ x\in\Om. 
\end{equation}
\par
The function $h$ defined in \eqref{def:h} satisfies $h\in \mathcal D_\lm$. 
To see that $h$ verifies the $(i)$-condition it is sufficient to note that 
\[
\int_\Om e^{\bal h}dx \geqslant |\Om|-\pi \ee_0>0.
\]
Then the claim follows by \cite{DJLW}. 
\end{proof}
%
%$$************************************************$$
%\textit{Step 1: For all $\la\in(\frac{8\pi}{\tau}, \barla)$ the set $\mathcal D_\lambda$ is not empty}
%
%\par\noindent
%\textit{Step 2: Existence of a minimax value}
%
Define
\begin{equation}\label{clambda}
c_\lm:=\inf_{h\in\mathcal D_\lm}\sup_{(r,\theta)\in\mathbb D} J_\lm(h(r,\theta)).
\end{equation}
We shall prove that $c_\lm$ defines a critical value for $J_\lm$ using the Struwe Monotonicity Trick contained in \cite{RicZecprep}, Proposition 4.1.
\par\noindent
In view of Lemma~\ref{lemma-dom-1}, we have $c_\lm<+\infty$. 
\begin{lemma}
\label{lemma-dom-2}
For any $\lm\in ({8\pi},\barlm)$, $c_\lm >-\infty.$
\end{lemma}
\begin{proof}
Denote by $B$ a bounded component of $\Rd\setminus \Om$ with at least an interior point and such that $\Gamma_1$ encloses $B$.
By continuity and by the $(iii)$-property defining 
$\mathcal D_\lm$, we have $\cm(h(\mathbb D))\supset B$ for all $h\in\mathcal D_\lm$.
By contradiction, assume  that $c_\lm=-\infty$. Then, there exists a sequence
$\{h_n\}\subset\mathcal D_\lm$ such that $\sup_{(r,\theta)\in \mathbb D} J_\lm(h_n(r,\theta))\rightarrow -\infty.$ 
Let $x_0$ be an interior point of $B$.  
For every $n$ we take $(r_n,\theta_n)\in \mathbb D$ such that $\cm(h_n(r_n,\theta_n))=x_0$.
In view of Lemma~\ref{baricentri}, it should be $x_0\in \stackrel{o}{B}\cap\overline\Om=\emptyset$,
a contradiction.
\end{proof}
At this point we set
\begin{equation}
\label{g0}
\mathcal G(u)=  \ln\int_{\Om }e^{ u}\,dx+ \ts \ln \int_\Om e^{\bal u}\,dx
\end{equation}
so that our functional \eqref{eq:Jdef} takes the form
\[
J_\lm (u)= \frac{1}{2}\int_{\Om }\left\vert \nabla u\right\vert
^{2}-\lm \mathcal G(u).
\] 
\begin{lemma}\label{lemma-dom-3}
For ${8\pi}<\lm_1 \leqslant \lm_2<{16\pi}$,  we have $\mathcal D_{\lm_1}\subseteq \mathcal D_{\lm_2}$.
\end{lemma}
\begin{proof}
It is sufficient to note that whenever  $J_\lm  (u)\leqslant 0$ it is $\mathcal G(u) \geqslant 0,$ with $\mathcal G$ given by \eqref{g0} and this implies that 
\[ 
 J_{\lm_1}(u)\geqslant J_{\lm_2}(u)  \qquad \mbox{ for } \quad {8\pi}<\lm_1 \leqslant \lm_2<{16\pi}\qquad \mbox{ if } J_{\lm_1} (u)\leqslant 0.%\lambda_{\bal.\tau}.
\]
Hence, $\mathcal D_{\lm_1}\subseteq \mathcal D_{\lm_2}$ for every ${8\pi}<\lm_1 \leqslant \lm_2<{16\pi}.$%\lambda_{\bal.\tau}$.
\end{proof}
\begin{lemma}
\label{proprg-uno}
The function $\mathcal G:H_0^1(\Om)\rightarrow \R$ defined by \eqref{g0} satisfies:
\begin{itemize}
\item[1)]{} $\mathcal G\in\mathcal C^2 (H_0^1(\Om);\R)$ 
\item[2)] $\mathcal G'$ is\ compact
\item[3)] $\left \langle \mathcal{G''}(u)\fip,\fip\right \rangle\geqslant0$ for every $ u,\fip\in H^1_0(\Om),
$ where $\langle\cdot,\cdot\rangle$  is the $L^2$-inner product.
\end{itemize} 
\end{lemma}
\begin{proof}
For every $u,\varphi\in H_0^1(\Om)$ we have:
\begin{equation*}
\mathcal G^{\prime}(u)\varphi= \frac{\int _\Om \varphi e^{u}dx}{\int_{\Om }e^{ u}dx}+ \ts \frac{\int _\Om \bal\varphi e^{\bal u}dx}{\int_{\Om }e^{ \bal u}dx}
\end{equation*}
and therefore the compactness of $\mathcal G'$ follows by the compactness of the Moser-Trudinger embedding. 
Moreover, for every $ u,\fip\in H^1_0(\Om) $ we have, using the Schwarz inequality,
\begin{equation*}
\begin{split}
\left\langle \mathcal G^{\prime\prime }(u)\varphi ,\varphi \right\rangle
&=\frac{1}{\left(\int_{\Om} e^{ u}dx \right)^2}\left[  \left(\int_{\Om} e^{ u}\fip^2  dx \right) \left(\int_{\Om} e^{ u}  dx \right)
-\left(\int_{\Om} e^{ u}\fip  dx \right)^2
\right]\\
&+\frac{\bal ^{2}\ts }{\left(\int_{\Om} e^{\bal  u}dx \right)^2} \left[ \left(\int_{\Om} e^{\bal u}\fip^2  dx \right) 
\left(\int_{\Om} e^{\bal u}  dx \right)
-\left(\int_{\Om} e^{\bal u}\fip  dx \right)^2
\right]\geqslant 0.
\end{split}
\end{equation*}
\end{proof}
Now we are able to prove the following.
\begin{prop}\label{prop:criticalvalueonOmega}
Let $\ts>0$ and assume that \eqref{def:gammasigma} holds. For almost every $\lm\in( {8\pi}, \barlm),$ $c_\lm>-\infty$ 
given by \eqref{clambda} is a saddle-type critical value for $J_\lm$.
\end{prop}
\begin{proof}[Proof of Proposition~\ref{prop:criticalvalueonOmega}]
In view of Lemma~\ref{proprg-uno}, Lemma~\ref{lemma-dom-1}, Lemma~\ref{lemma-dom-2} and Lemma~\ref{lemma-dom-3},
we may apply the well known Struwe's monotonicity trick to derive the existence of the desired critical value. 
See \cite{DJLW} or \cite{RicZecprep}, Proposition~4.1  
with $\mathcal H=H_0^1(\Om)$,
$V=\mathbb D$, $A=-\infty$ and  $\mathcal F_\lambda=\mathcal D_\lm$. 
\end{proof}
\begin{proof}[Proof of Theorem~\ref{Intro:thm:existence}]\textit{(Completion by blow-up results).} 
We fix $\lm_0 \in ({8\pi},\barlm)$. In view of Proposition \ref{prop:criticalvalueonOmega} 
there exists $\lm_n\to\lm_0$ such that problem \eqref{Intro:specialpb} with $\lm=\lm_n$ admits 
a solution $u_n.$ By the blow up analysis as stated in Proposition \ref{Prop:masszero},  
we have the compactness of solution sequences. Therefore, up to subsequences, 
we obtain that $u_n\to u_0$  with $u_0$ a solution to \eqref{Intro:specialpb} with $\lm=\lm_0$.
\end{proof}
%%%%%%%%%%%%%%%%%%%%%%%%%%%%%%%%%%%%%%%%%%%%%%%%%%%%%%%%%%%%%%%%%%%%%%%%%%%%%%%%%%%
%%%%%%%%%%%%%%%%%%%%%%%%%%%%%%%%%%%%%%%%%%%%%%%%%%%%%%%%%%%%%%%%%%%%%%%%%%%%%%%%%%%
%%%%%%%%%%%%%%%%%%%%%%%%%%%%%%%%%%%%%%%%%%%%%%%%%%%%%%%%%%%%%%%%%%%%%%%%%%%%%%%%%%%
%%%%%%%%%%%%%%%%%%%%%%%%%%%%%%%%%%%%%%%%%%%%%%%%%%%%%%%%%%%%%%%%%%%%%%%%%%%%%%%%

\end{document}